\documentclass{amsart}

\usepackage{amsrefs}
\usepackage{amsmath}
\usepackage{bbm}
\usepackage{amsfonts}
\usepackage[hidelinks]{hyperref}
\usepackage{amssymb}
\usepackage{enumitem}
\usepackage{todonotes}

\DefineSimpleKey{bib}{primaryclass}{}
\DefineSimpleKey{bib}{archiveprefix}{}
\BibSpec{arXiv}{
  +{}{\PrintAuthors}{author}
  +{,}{ \textit}{title}
  +{}{ \parenthesize}{date}
  +{,}{ arXiv }{eprint}
}

\newtheorem{theorem}{Theorem}[section]
\newtheorem*{theorem*}{Theorem}
\newtheorem{proposition}[theorem]{Proposition}
\newtheorem{corollary}[theorem]{Corollary}
\newtheorem{lemma}[theorem]{Lemma}
\newtheorem{claim}{Claim}[theorem]

\theoremstyle{definition}
\newtheorem{definition}[theorem]{Definition}
\newtheorem*{conj}{Conjecture}

\theoremstyle{remark}
\newtheorem{remark}[theorem]{Remark}
\newtheorem{fact}[theorem]{Fact}

\begin{document}

\title{A GAME FOR BAIRE'S GRAND THEOREM}

\author{Lorenzo Notaro}
\address{Università degli Studi di Torino, Dipartimento di Matematica ``G. Peano", Via Carlo Alberto 10, 10123 Torino, Italy}
\curraddr{}
\email{lorenzo.notaro@unito.it}
\date{}
\thanks{The author would like to thank Rapha\"{e}l Carroy and Alessandro Andretta for their careful readings and useful suggestions, which helped shaping this article. The author would like to acknowledge INDAM for the financial support. This research was also partially supported by the project PRIN 2017 ``Mathematical Logic: models, sets, computability", prot. 2017NWTM8R}

\subjclass[2020]{Primary 03E15, Secondary 26A21, 91A44}
\keywords{Baire class 1, Baire characterization theorem, topological game, reduction game}

\begin{abstract}
Generalizing a result of Kiss, we provide a game that characterizes Baire class $1$ functions between arbitrary separable metrizable spaces. We show that the determinacy of our game  is equivalent to a generalization of Baire's grand theorem, and that both these statements hold under $\mathsf{AD}$ and in Solovay's model.
\end{abstract}

\maketitle

\section{Introduction}
\label{sec:1}
A Polish space is a separable, completely metrizable topological space. Given two topological spaces $X,Y$, a function $f:X\rightarrow Y$ is said to be Baire class $1$ if, for every open subset $V$ of $Y$, the pre-image $f^{-1}(V)$ is an $F_\sigma$ subset of $X$, i.e. a countable union of closed subsets. If $X$ is metrizable, then the open subsets of $X$ are also $F_\sigma$ subsets. All continuous functions with metrizable domain are Baire class $1$.

A classical result concerning this class of functions is the following theorem of Baire --- known as Baire's grand theorem --- which provides a characterization of Baire class $1$ functions from a Polish space to a separable metrizable space (e.g. see \cite{CDST}*{Theorem 24.15}).
\begin{theorem*}[Baire]
\label{BGT}
Let $X$ be a Polish space, $Y$ a separable metrizable space, and $f: X \rightarrow Y$. Then the following are equivalent:
\begin{enumerate}[label=\arabic*)]
    \item $f$ is Baire class $1$
    \item $f{\upharpoonright} K$ has a point of continuity for every compact $ K\subseteq X$
\end{enumerate}
\end{theorem*}

Actually, the separability hypothesis on $X$ can be avoided \citelist{\cite{Kuratowski}*{Ch. II, \S 31, X}\cite{GrandBaire}*{Corollary 1}}, but in this article we are interested in the separable case. 

We note that Baire class $1$ functions have been, and still are, sometimes defined as pointwise limits of continuous functions --- e.g. \cites{Hausdorff,MR3819690, MR377482}, Baire himself originally stated his grand theorem for pointwise limits of continuous real functions \cite{Baire}. This definition and ours are equivalent only under certain hypotheses --- e.g. see \citelist{\cite{CDST}*{Theorem 24.10} \cite{Karlova2}}. 

In this article, we study the generalization of \hyperref[BGT]{Baire's grand theorem} in which the domain's hypothesis is weakened from Polish to separable metrizable, and its relationship with the determinacy of a two-player game. The use of infinite two-player, perfect information games to characterize certain classes of functions has a long and established history --- e.g. \cites{Wadge,Duparc,Andretta,Carroy1,Kiss,Camerlo, Limsup, balcerzak}, see \cite{MottoRos} for a detailed introduction on this subject.

In \hyperref[sec:2]{Section 2}, we define our game $G(f)$, where $f$ is a function between separable metrizable spaces. We prove that Player II has a winning strategy in $G(f)$ if and only if $f$ is Baire class $1$ (Theorem \ref{PlayerII}). Then we show that Player I has a winning strategy in $G(f)$ if and only if there is a compact $ K\subseteq X$ such that $f{\upharpoonright} K$ has no points of continuity (Theorem \ref{PlayerI}). 

In \hyperref[sec:3]{Section 3}, we discuss the determinacy of our game. We start by observing that the determinacy of our game for every function is equivalent to \hyperref[GBT]{$\mathsf{GBT}$}, the generalization of \hyperref[BGT]{Baire's grand theorem} in which the domain's hypothesis is weakened from Polish to separable metrizable (Corollary \ref{equiv1}). We note that $\mathsf{AC}$ and $\mathsf{GBT}$ are mutually inconsistent (Proposition \ref{ZFC}). Then we show that  \hyperref[GBT]{$\mathsf{GBT}$} is equivalent to a separation property coming from Descriptive Set Theory (Theorem \ref{equiv2}) and that both these statements hold under $\mathsf{AD}$, the axiom of determinacy, and in Solovay's model.

\section{The game}
\label{sec:2}
Given $X$ a topological space and $(U_n)_{n\in\mathbb{N}}$ a sequence of open subsets of $X$, we say that $(U_n)_{n\in\mathbb{N}}$ is \emph{convergent} if it is decreasing with respect to $\subseteq$, and if it is a local basis of some  $x \in X$. In that case we write $\lim_{n\rightarrow \infty} U_n = x$.

\begin{definition}\label{definition}
    Let $X,Y$ be separable metrizable spaces and let $f:X\rightarrow Y$. In our game $G(f)$, at the $n$th round, Player I plays a nonempty open subset $U_n$ of $X$,  and then Player II plays $y_n\in \text{ran}(f)$,
\begin{center}
\begin{tabular}{cccccccc}
I & $U_0$ & & $U_1$ & & $U_2$ & & ... \\
II & & $y_0$ & & $y_1$ & & $y_2$ & ...\\
\end{tabular}
\end{center}
\noindent with the rule: $U_{n+1}\subseteq U_{n}$ for each $n\in\mathbb{N}$. At the end of a game run, Players I and II have produced a sequence $(U_n)_{n\in\mathbb{N}}$ of nonempty open subsets of $X$ and a sequence $(y_n)_{n\in\mathbb{N}}$ in $\text{ran}(f)$, respectively. Player II \emph{wins} the run if either the sequence $(U_n)_{n\in\mathbb{N}}$ is not convergent or it converges to an $x \in X$ and $(y_n)_{n\in\mathbb{N}}$ converges to $f(x)$.
\end{definition}

 This game is an elaboration of Kiss' game \cite{Kiss} and a further generalization of Duparc's eraser game \cite{Duparc}. 

Since we use trees and operations over finite sequences throughout, we briefly recall their classical definition and notation. Let $X$ be a nonempty set. We denote by $X^{<\mathbb{N}}$ the set of finite sequences of elements of $X$, with $\emptyset$ being the empty sequence. Let $u\in X^{<\mathbb{N}}$ and $v \in X^{\le\mathbb{N}}$. We write $u \sqsubseteq v$ when $u$ is an initial segment of $v$, and we say that $v$ \emph{extends} $u$. We call $u^\smallfrown v$ the concatenation of $u$ and $v$. For each $n \le \text{lenght}(u)$, we let $u {\upharpoonright} n$ be the initial segment $\langle u(0), u(1), \dots, u(n-1)\rangle$ if $n>0$; otherwise, we let it be the empty sequence. Given an $x \in X$, we write $u^\smallfrown x$ to denote the finite sequence $u ^\smallfrown \langle x \rangle$, where $\langle x \rangle$ is the sequence of length $1$ containing only $x$. Sometimes we use the notation {$\vec{s}$} to denote a sequence. A \emph{tree} $T$ on $X$ is a subset of $X^{<\mathbb{N}}$ closed under initial segments. A tree is said to be \emph{pruned} if all its nodes have a proper extension. Given a tree $T$, we denote by $[T]$ the set $\{r \in X^\mathbb{N} \mid \forall n \  r {\upharpoonright} n \in T\}$, whose elements are called \emph{branches} of $T$. The family $\{N_s \mid s \in \mathbb{N}^{<\mathbb{N}}\}$ with $N_s = \{r \in \mathbb{N}^\mathbb{N} \mid s \sqsubseteq r\}$, is the standard basis of the Baire space.

A \emph{partial play} of Player I (resp. Player II) in $G(f)$ is a nonempty finite sequence of open subsets of $X$ decreasing with respect to $\subseteq$ (resp.\ a finite sequence of elements of $\text{ran}(f)$). A \emph{strategy} for Player I (resp. Player II) in $G(f)$ is a function that maps each partial play of Player II (resp. Player I) to an open subset of $X$ (resp. to an element of $\text{ran}(f)$). A strategy $\sigma$ for Player I is \emph{winning} if for every infinite sequence $\vec{y} = (y_n)_{n\in\mathbb{N}}$ of elements of $\text{ran}(f)$, Player I wins the run of the game in which at each turn $n$ Player I plays $\sigma(\vec{y} \upharpoonright n)$ and Player II plays $y_n$. Winning strategies for Player II are analogously defined.

Kiss \cite{Kiss}*{Theorem 1} used his game to characterize Baire class $1$ functions between separable complete metric spaces. The following theorem generalizes this result by providing an analogous characterization for Baire class $1$ functions  between arbitrary separable metrizable spaces.
\begin{theorem}\label{PlayerII}
Let $X, Y$ be separable metrizable spaces and $f:X\rightarrow Y$. Then Player II has a winning strategy in $G(f)$ if and only if $f$ is Baire class $1$. 
\end{theorem}
\begin{proof}
$(\Longleftarrow)$: Assume that the function $f$ is Baire class $1$. As every separable metrizable space embeds into $\mathbb{R}^\mathbb{N}$ (i.e. the space of infinite sequences of real numbers with the product of the Euclidean topology), we can assume $Y\subseteq \mathbb{R}^\mathbb{N}$ without loss of generality. 

By a theorem of Lebesgue, Hausdorff and Banach \cite{CDST}*{Theorem 24.10} there exists a sequence $(f_n)_{n\in\mathbb{N}}$ of continuous functions from $X$ to $\mathbb{R}^\mathbb{N}$ (with range not necessarily in $Y$) converging pointwise to $f$. Fix such a sequence, and also fix a compatible metric $d$ on $\mathbb{R}^\mathbb{N}$ and a sequence $(q_n)_{n\in\mathbb{N}}\subset \text{ran}(f)$ dense in $\text{ran}(f)$. Given two nonempty $A,B\subseteq \mathbb{R}^\mathbb{N}$, we let $d(A,B)$ be $\inf \{d(z_0,z_1) \mid z_0 \in A, z_1 \in B\}.$

In the next paragraph, we define by induction a map $\sigma^\star$ that maps partial plays of Player I into $Y\times \mathbb{N}$. Then $\pi_Y\circ \sigma^\star$ is shown to be a winning strategy for Player II, where $\pi_Y, \pi_\mathbb{N}$ are the canonical projections. We denote by $\sigma^\star_Y$ and $\sigma^\star_\mathbb{N}$ the functions $\pi_Y\circ \sigma^\star$ and $\pi_\mathbb{N}\circ \sigma^\star$, respectively.

Here is the definition on $\sigma^\star$ by induction on the lengths of Player I's partial plays: set $\sigma^\star ( U_0 ) = (q_0, 0)$ for each nonempty open subset $U_0$ of $X$; fix $k\in\mathbb{N}$, suppose that we have defined $\sigma^\star$ for all Player I's partial plays of length up to $k$ and consider a partial play $\vec{U}^\smallfrown U_{k}$ of length $k+1$, then
\begin{enumerate}[label={\arabic*)}]
    \item if there is an $n > \sigma^\star_\mathbb{N}(\vec{U})$ such that $\text{diam}(f_{n}[U_k]) \le  2^{-n}$: fix an $n$ satisfying the condition and an $m$ such that $d(q_m, f_n[U_k]) \le d(\text{ran}(f), f_{n}[U_{k}])+ 2^{-n}$; set $\sigma^\star(\vec{U}^\smallfrown U_k) = (q_m, n)$.
    \item otherwise: we set $\sigma^\star(\vec{U}^\smallfrown U_{k}) = \sigma^\star(\vec{U})$.
\end{enumerate} 

\noindent We now show that $\sigma^\star_Y$ is a winning strategy for Player II in $G(f)$. Fix an infinite play $(U_k)_{k\in\mathbb{N}}$ of Player I in $G(f)$. If $(U_k)_{k\in\mathbb{N}}$ is not convergent, then Player II wins. Assume that $(U_k)_{k\in\mathbb{N}}$ converges to an $x \in X$, and set $y_k = \sigma^\star_Y (U_0,\dots, U_k), \ n_k = \sigma^\star_\mathbb{N} (U_0,\dots, U_k)$ for each $k \in \mathbb{N}$. We now need to show $\lim_{k\rightarrow \infty} y_k = f(x)$.

\begin{claim}
The sequence $(n_k)_{k\in\mathbb{N}}$ is nondecreasing and unbounded in $\mathbb{N}$.
\end{claim}
\begin{proof}
The fact that $(n_k)_{k\in\mathbb{N}}$ is nondecreasing is a direct consequence of the definition of $\sigma^\star$. Next, note that, for all $n \in \mathbb{N}$, the diameters of the sets in the sequence $(f_n[U_k])_{k\in\mathbb{N}}$ converge to 0, as $f_n$ is continuous and $(U_k)_{k\in\mathbb{N}}$ is a local basis of $x$, decreasing with respect to $\subseteq$. 

Fix a $k \in \mathbb{N}$ and an $n > n_k$. By the previous observation, there exist a $k' > k$ such that $\text{diam}(f_{n}[U_{k'}]) \le 2^{-n}$. Fix one such $k'$, there are two cases: either $n_{k'-1} > n_k$ or $n_{k'-1} = n_k$. In the latter case, the first condition in the inductive definition of $\sigma^\star$ happens at the $k'$-th round, hence $n_{k'} > n_{k'-1} = n_k$. In either case, $n_{k'} > n_k$. We just proved that for every $k$ there is a $k'>k$ such that $n_{k'} > n_k$, therefore $(n_k)_{k\in\mathbb{N}}$ is unbounded.
\end{proof}
Let $\bar{k}$ be the least $k$ such that $n_k > 0$. 
\begin{claim}
    For all $k \ge \Bar{k}, \ d(y_k, f_{n_k}(x)) \le d( f(x), f_{n_k}(x)) + 2^{1-n_k}$.
\end{claim}
\begin{proof}
  Fix a $k \ge \Bar{k}$ and pick the smallest $l \le k$ such that $n_l = n_k$. Note that $y_k = y_l$, as from the $l$-th round to the $k$-th $\sigma^\star$ does not change its response. From the minimality of $l$ it follows that the first condition of the inductive definition of $\sigma^\star$ happens at the $l$-th round, therefore $d(y_l, f_{n_l}[U_l]) \le d(\text{ran}(f), f_{n_l}[U_l]) + 2^{- n_l}$ and $\text{diam}(f_{n_l}[U_l]) \le 2^{-n_l}$. Since we assumed $x = \lim_{n \rightarrow \infty} U_n$, $x$ belongs to $U_l$ and $f_{n_l}(x)$ belongs to $f_{n_l}[U_l]$, hence $d(\text{ran}(f), f_{n_l}[U_l]) \le d(f(x), f_{n_l}(x))$, and, overall, $$d(y_l, f_{n_l}(x)) \le d(y_l, f_{n_l}[U_l]) + \text{diam}(f_{n_l}[U_l]) \le d(f(x), f_{n_l}(x)) +2^{1-n_l}.$$ As $n_l = n_k$ and $y_l = y_k$ we are done.
\end{proof}

Then, for each $k \ge \bar{k}$,
 \begin{align*}
    d(y_k, f(x)) \le d(f_{n_k}(x), f(x)) +  d(y_k, f_{n_k}(x))  \le 2 d(f_{n_k}(x), f(x)) + 2^{1-n_{k}}.
\end{align*}

\noindent Since $(n_k)_{k\in\mathbb{N}}$ is unbounded and the $f_n$'s converge pointwise to $f$, these inequalities imply that $(y_k)_{k\in\mathbb{N}}$ converges to $f(x)$ and therefore $\sigma^\star_Y$ wins the run. As $(U_k)_{k\in\mathbb{N}}$ was an arbitrary play of Player I, we have shown that $\sigma^\star_Y$ is a winning strategy for Player II in $G(f)$.\vspace{1em}

$(\Longrightarrow)$: Suppose that Player II has a winning strategy in $G(f)$, we show that the function $f$ is Baire class $1$.

Fix a winning strategy $\sigma$ for Player II in $G(f)$ and fix a compatible metric $d$ on $X$. As $X$ is separable, there exists a scheme $(U_s)_{s \in \mathbb{N}^{<\mathbb{N}}}$ of open subsets of $X$ satisfying the following properties:
\begin{enumerate}[label=\arabic*)]
    \item $U_{\emptyset} = X$.
    \item For all $s\in\mathbb{N}^{<\mathbb{N}},\  \bigcup_n U_{s^\smallfrown  n}= U_s$.
    \item For all $s\in\mathbb{N}^{<\mathbb{N}},\ \text{diam}(U_{s}) \le 2^{- \text{length}(s)}$.
\end{enumerate}
For each $s \in \mathbb{N}^{<\mathbb{N}}$ let $$y_s=\sigma(U_{s\upharpoonright 0},U_{s\upharpoonright 1},U_{s\upharpoonright 2},\dots, U_s).$$ In other words, $y_s$ is the response of Player II following $\sigma$ to the partial play $(U_{s\upharpoonright 0},\dots, U_s)$ of Player I. For every $x \in X$, let $T_x$ be the tree $\{s \in \mathbb{N}^{<\mathbb{N}} \mid x \in U_s\}$. It follows directly from properties 1) and 2) of the scheme that all such trees are nonempty and pruned.

\begin{claim}\label{Konig}
For all $x\in X$ and all open neighborhoods $V$ of $f(x)$, there is an $s \in T_x$ such that for all $t \in T_x$ if $ s \sqsubseteq t$ then $y_t \in V$.
\end{claim}
\begin{proof}
Assume for a contradiction that there is an $x\in X$ and a $V$, open neighborhood of $f(x)$, such that for all $s \in T_x$ there is a $t \in T_x$ that extends $s$ and such that $y_t \not\in V$. Then there exists a branch $r \in [T_x]$  such that $\{n \in \mathbb{N} \mid y_{r{\upharpoonright} n} \not\in V\}$ is infinite. Fix one and note that Player I wins in $G(f)$ by playing the sequence $(U_{r\upharpoonright n})_{n\in\mathbb{N}}$, as, by property 3) of the scheme, this sequence converges to $x$, while the corresponding play of Player II according to $\sigma$ does not converge to $f(x)$. Since we have assumed $\sigma$ to be a winning strategy for Player II, we have reached a contradiction.
\end{proof}
 
Fix $V$ open subset of $Y$ and a sequence $(V_n)_{n\in\mathbb{N}}$ of open subsets of $V$ such that $V = \bigcup_{n\in\mathbb{N}} V_n = \bigcup_{n\in\mathbb{N}} \overline{V_n}$. 
\begin{claim}
$$f^{-1}(V) = \bigcup_{n \in \mathbb{N}} \  \bigcup_{s \in \mathbb{N}^{<\mathbb{N}}} \left( U_s \setminus \bigcup\left\{U_t \mid t \sqsupseteq s \  \mathrm{and}  \ y_t \not\in V_n\right\}\right).$$
\end{claim}
\begin{proof}
Take  $x$ in the set on the right-hand side. By definition, there exists an $n \in \mathbb{N}$ and an $s \in \mathbb{N}^{<\mathbb{N}}$ such that $s \in T_x$ and for all $t \in T_x$ extending $s, \ y_t \in V_n$.  Fix a branch $r \in [T_x] \cap N_s$, then the sequence $(y_{r \upharpoonright k})_{k\in\mathbb{N}}$ is eventually in $V_n$, i.e. $y_{r \upharpoonright k} \in V_n$  for all $k$ greater than some $m \in \mathbb{N}$. As $(U_{r{\upharpoonright}k})_{k\in\mathbb{N}}$ converges to $x$ by property 3) of the scheme, and $\sigma$ is a winning strategy for Player II, we have $\lim_{k \rightarrow \infty}y_{r \upharpoonright k} = f(x)$, and therefore $f(x) \in \overline{V_n}\subseteq V$.

Now pick an $x$ in $f^{-1}(V)$. There must be an $n$ such that $f(x) \in V_n$. By Claim \ref{Konig}, there exists an $s \in \mathbb{N}^{<\mathbb{N}}$ such that for all $t \in T_x$ extending $s$, $y_t \in V_n$. But this means that $x$ belongs to the set on the right-hand side.
\end{proof}

Since the set on the right-hand side is an $F_\sigma$ subset of $X$ and $V$ was an arbitrary open subset of $Y$, we have shown that pre-images of open subsets of $Y$ by $f$ are $F_\sigma$ sets. So $f$ is Baire class $1$.
\end{proof}

\begin{theorem}\label{PlayerI}
Let $X, Y$ be separable metrizable spaces and $f:X\rightarrow Y$. Then Player I has a winning strategy in $G(f)$ if and only if there exists a compact set $ K\subseteq X$ such that $f{\upharpoonright} K$ has no points of continuity.
\end{theorem}
To prove this theorem, we need the notion of pointwise oscillation of a function. Given $X$ a topological space, $Y$ a metric space and $f:A \rightarrow Y$ for some nonempty $A\subseteq X$, we define $\mathrm{osc}_f(x)$ for each $x \in X$ as $$\mathrm{osc}_{f}(x) = \inf \{\text{diam}(f(U \cap A)) \mid U\subseteq X \text{ open neighborhood of } x\}.$$
The function $\mathrm{osc}_{f}$ is upper semi-continuous, i.e. for every $\epsilon > 0$ the set $\{x \in X \mid \mathrm{osc}_{f}(x) \ge \epsilon\}$ is closed.
 
\begin{lemma}\label{osc}
Let $X,Y$ be separable metric spaces, $\epsilon > 0$ and $f:X\rightarrow Y$ such that $\mathrm{osc}_f(x) \ge \epsilon$ for all $x \in X$. Then there is a countable $Q \subseteq X$ such that $\mathrm{osc}_{f \upharpoonright Q}(x) \ge \epsilon$ for all $x \in X$. 
\end{lemma}
\begin{proof}
Let $d_Y$ be the metric on $Y$ and fix a sequence $(y_n)_{n\in\mathbb{N}}$ dense in $Y$. For each $n,m \in \mathbb{N}$, let $Q_{n,m}$ be a countable and dense subset of $f^{-1}(B(y_n,2^{-m}))$. We claim that the countable set $Q = \bigcup_{n,m} Q_{n,m}$ has the desired property.

Fix $x \in X$, $m \in \mathbb{N}$ and an open neighborhood $U$ of $x$. By assumption, there are $x_0, x_1 \in U$ such that $d_Y(f(x_0), f(x_1)) \ge \epsilon - 2^{-m}$. Let $n_0,n_1$ be such that $f(x_i) \in B(y_{n_i}, 2^{-m})$ for $i = 0,1$.  In particular, $U \cap f^{-1}(B(y_{n_i}, 2^{-m})) \neq \emptyset$, and therefore $U \cap Q_{n_i,m} \neq \emptyset$ for $i=0,1$. Pick $q_0, q_1$ in $U \cap Q_{n_0,m}$ and $U \cap Q_{n_1,m}$, respectively. Then,
\begin{align*}
    d_Y(f(q_0), f(q_1)) &\ge d_Y(f(x_0), f(x_1)) - d_Y(f(x_0), f(q_0)) - d_Y(f(x_1), f(q_1))\\ &\ge (\epsilon - 2^{-m}) - 2^{1-m} - 2^{1-m} = \epsilon - 5\cdot 2^{-m}.
\end{align*}
Indeed, for $i=0,1$, $f(x_i)$ and $f(q_i)$ both belong to $B(y_{n_i}, 2^{-m})$, and therefore their distance is less or equal to $2^{1-m}$.

We have shown that for each $x \in X$, for every open neighborhood $U$ of $x$ and for all $m$ there are $q_0, q_1 \in U \cap Q$ such that $d_Y(f(q_0),f(q_1))$ is greater than $\epsilon - 5\cdot 2^{-m}$. In particular, $\text{diam}(f(U \cap Q)) \ge \epsilon$. Hence, for all $x \in X, \ \mathrm{osc}_{f{\upharpoonright} Q} (x) \ge \epsilon$.  
\end{proof}

\begin{proof}[Proof of Theorem \ref{PlayerI}]
$(\Longleftarrow):$ Fix a compact set $ K\subseteq X$ such that $f{\upharpoonright}  K$ has no points of continuity. The winning strategy for Player I that we define is essentially the one defined by Kiss\footnote{Kiss' strategy, in turn, is based on the one defined by Carroy in \cite{Carroy1}*{Theorem 4.1}} in \cite{Kiss}*{\S 2}, the only difference being that we deal with a bit more care the amount of choice used in the construction (see Remark \ref{choice}).

Fix a compatible metric $d_X$ on $X$ and $d_Y$ on $Y$.
Since $f{\upharpoonright}  K$ has no points of continuity, it follows that $\mathrm{osc}_{f\upharpoonright  K}(x) > 0$ for every $x \in K$. In particular, $ K = \bigcup_n  K_n$, where $$ K_n = \left\{x \in  K \mid \mathrm{osc}_{f\upharpoonright K}(x) \ge \frac{1}{n}\right\}.$$

By Baire's category theorem, there is a nonempty open $U\subseteq X$ and an $n$ such that $ K_n \cap U =  K \cap U$. Let $C$ be the closure of $ K_n \cap U$ and $\epsilon = 1/n$, then $\mathrm{osc}_{f\upharpoonright C}(x) \ge \epsilon$ for every  $x \in C$.

By Lemma \ref{osc}, we know that there is a countable $Q\subseteq C$ such that $\mathrm{osc}_{f\upharpoonright Q}(x) \ge \epsilon$ for every $x \in Q$. Let $(q_n)_{n\in\mathbb{N}}$ be an enumeration of $Q$. We now define a winning strategy $\tau$ for Player I by induction on the lengths of Player II's partial plays. In particular, the map $\tau$ ranges among the open balls of $X$ centered in $Q$, i.e. open sets of the form $B(x,\rho)$ for some $x\in Q$ and radius $\rho>0$: first set $\tau(\emptyset) = B(q_0, 1)$ --- we are setting the first move of Player I; fix $k \in \mathbb{N}$, suppose that we have defined $\tau$ for all partial plays of Player II of lengths up to $k$ and consider the partial play $\vec{y}\ ^\smallfrown y_k$ of length $k+1$ with $B(q_{n_k}, \rho_k) = \tau(\vec{y})$, then

    \begin{enumerate}[label={\arabic*)}]
    \item if $d_Y(y_k, f(q_{n_k})) \le \epsilon/8$: \\ let $n_{k+1}$ be an $n$ such that $q_n \in B(q_{n_k}, \rho_k)$ and $d_Y(f(q_n), f(q_{n_k})) \ge \epsilon/3$; let $\overline{\rho}$ be the greatest $\rho \le \rho_k$ such that $B(q_{n_{k+1}}, \rho) \subseteq B(q_{n_k}, \rho_k)$ and set $\tau(\vec{y}\ ^\smallfrown y_k) = B(q_{n_{k+1}}, \overline{\rho}/2)$.
    \item otherwise:\\ $\tau(\vec{y}\ ^\smallfrown y_k) = B(q_{n_k}, \rho_k / 2)$.
\end{enumerate}

\noindent We now prove that $\tau$ is a winning strategy for Player I. Fix an infinite play $\vec{y} = (y_k)_{k\in\mathbb{N}}$ of Player II and set $B_k = B(x_k, \rho_k) = \tau(\vec{y}\upharpoonright k)$ for every $k\in\mathbb{N}$. First we show that the sequence $(B_k)_{k\in\mathbb{N}}$ converges to an $x \in  K$.  Indeed, it follows directly from $\tau$'s inductive definition that $\bigcap_k B_k = \bigcap_k \overline{B_k}$; the compactness of $K$ guarantees that $ K\cap \bigcap_k \overline{B_k} \neq \emptyset$; finally, the radii of $(B_k)_{k\in\mathbb{N}}$ converge to $0$, hence $ K \cap \bigcap_k {B_k}$ is a singleton $\{x\}$ and $(B_k)_{k\in\mathbb{N}}$ converges to $x$.

So we are left to prove that the sequence $(y_k)_{k\in\mathbb{N}}$ does not converge to $f(x)$. Suppose first that condition 1) of $\tau$'s inductive definition happens only finitely many times during this game run. This means that there exists an $n$ such that for all $k\ge n, \ x_k = x$, and therefore $d_Y(y_k, f(x)) > \epsilon/8$ for all $k \ge n$. In this case $(y_k)_{k\in\mathbb{N}}$ certainly does not converge to $f(x)$.

Now suppose otherwise, and let the increasing sequence $(k_n)_{n\in\mathbb{N}}$ be such that condition 1) happens at the $k_n{+}1$-th round for each $n$. More precisely, $(k_n)_{n\in\mathbb{N}}$ is the increasing sequence such that $d_Y(y_k,f(x_k))\le \epsilon/8$ if and only if $k = k_n$ for some (unique) $n$. For every $n$,
\begin{align*}
    d_Y(y_{k_n}, y_{k_{n+1}}) & \ge d_Y(f(x_{k_n}), f(x_{k_{n+1}})) - d_Y(f(x_{k_n}), y_{k_n}) - d_Y(f(x_{k_{n+1}}), y_{k_{n+1}}) \\ &= d_Y(f(x_{k_n}), f(x_{k_n+1}))- d_Y(f(x_{k_n}), y_{k_n}) - d_Y(f(x_{k_{n+1}}), y_{k_{n+1}})\\& \ge \epsilon/3 - \epsilon/8 - \epsilon /8 = \epsilon/12
\end{align*}
where the equality follows from $x_{k_{n+1}}=x_{k_n+1}$, which holds because in the rounds between $k_n+1$ and $k_{n+1}$ the strategy $\tau$ does not change the center of its balls; the last inequality follows directly from the definition of $\tau$.
Therefore, as $(k_n)_{n\in\mathbb{N}}$ is unbounded, the sequence $(y_k)_{k\in\mathbb{N}}$ does not converge. 

In either case $(y_k)_{k\in\mathbb{N}}$ does not converge to $f(x)$, therefore $\tau$ wins the run. As $(y_k)_{k\in\mathbb{N}}$ was an arbitrary play of Player II, we have shown that $\tau$ is a winning strategy for Player I in $G(f)$.\vspace{1em}

$(\Longrightarrow):$ Suppose that Player I has a winning strategy in $G(f)$, we want to prove that there exists a compact set $ K\subseteq X$ such that $f{\upharpoonright}  K$ has no points of continuity.

We show instead that there exists a compact $ K\subseteq X$ such that Player I has a winning strategy in $G(f{\upharpoonright}  K)$. Indeed, if we do so, it would mean that the function $f{\upharpoonright}  K$ is not Baire class $1$, as otherwise Player II would have a winning strategy in $G(f{\upharpoonright}  K)$ by Theorem \ref{PlayerII}. Then, by \hyperref[BGT]{Baire's grand theorem} --- which can be applied as $K$, being a compact separable metrizable space, is a Polish space --- there would be a compact $ K'\subseteq  K$ such that $f{\upharpoonright}  K'$ has no points of continuity.

Fix a winning strategy $\tau$ for Player I and also fix an  enumeration $(q_n)_{n\in\mathbb{N}}$ of a countable dense subset of $\text{ran}(f)$. Denote by $S$ the tree $\{s \in \mathbb{N}^{<\mathbb{N}} \mid s(n) \le n \text{ for all } n < \text{length}(s)\}$. Note that $[S]$ is a compact subset of the Baire space.

Consider the following map:
\begin{align*}
    \varphi: [S] &\longrightarrow X\\ r &\longmapsto \lim_{n\rightarrow \infty} \tau(q_{r(0)},q_{r(1)}, \dots, q_{r(n)}).
\end{align*}
Since we are assuming $\tau$ winning  for Player I, the limits in the definition always exist, and the map $\varphi$ is well-defined. We now show that $\varphi$ is continuous. Given an $r \in [S]$ and $V$ open neighborhood of $\varphi(r)$, there exists an $n\in \mathbb{N}$ such that $\tau(q_{r(0)},q_{r(1)}, \dots, q_{r(n-1)}) \subseteq V$, by definition of limit of sequences of open sets. But then the rules of the game force every $t \in [S] \cap N_{r{\upharpoonright} n}$ to be mapped by $\varphi$ into $\tau(q_{r(0)},q_{r(1)}, \dots, q_{r(n-1)}) \subseteq V$. Therefore $\varphi$ is continuous and $ K=\text{ran}(\varphi)$ is a compact subset of $X$.

Next, we show that Player I has a winning strategy in $G(f{\upharpoonright}  K)$. Fix $d_Y$ compatible metric on $Y$. For each $y \in Y$ and $k \in \mathbb{N}$, pick an $n \le k$ such that
 $d_Y(q_n,y) = \min_{m\le k} d_Y(q_m, y)$ and let $q(y,k) = q_n$.

We define the strategy $\tau'$ for Player I in $G(f{\upharpoonright}  K)$ as follows: for each $(y_0, \dots, y_k)$ partial play of Player II in $G(f{\upharpoonright} K)$, we let $$\tau'(y_0, y_1, \dots, y_k) = \tau(q(y_0,0), q(y_1,1), \dots, q(y_k,k)) \cap  K.$$

We claim that $\tau'$ is a winning strategy for Player I in $G(f{\upharpoonright}  K)$. Take an infinite play $(y_k)_{k\in\mathbb{N}}$ of Player II. For each $k$, let $n_k$ be such that $n_k \le k$ and $q_{n_k} = q(y_k,k)$. Then $(n_k)_{k\in\mathbb{N}}$ belongs to $[S]$, and the limit of the sequence $(\tau'(y_0,\dots, y_k))_{k \in \mathbb{N}}$ is $\varphi((n_k)_{k\in\mathbb{N}}) \in  K$, by definition of $\varphi$.

If  $(y_k)_{k\in\mathbb{N}}$ is not convergent, then Player I wins the run. So suppose that $(y_k)_{k\in\mathbb{N}}$ converges to $y \in \overline{\text{ran}(f)}$. Then,
 \begin{align*}
     d_Y(q(y_k,k), y) \le d_Y(y_k,y) + d_Y(q(y_k,k), y_k) &=  d_Y(y_k,y) + \min_{m\le k} d_Y(q_m, y_k) \\ &\le 2 d_Y(y_k,y) + \min_{m\le k} d_Y(q_m, y).
 \end{align*}
Since $\lim_{k\rightarrow \infty} y_k = y$ by assumption, and $\lim_{k\rightarrow \infty} \min_{m\le k} d_Y(q_m, y) = 0$ by the density of $(q_n)_{n\in\mathbb{N}}$ in $\text{ran}(f)$, it follows from the above inequalities that $(q(y_k,k))_{k\in\mathbb{N}}$ converges to $y$. Therefore,
$$\lim_{k\rightarrow \infty} y_k = \lim_{k \rightarrow \infty} q(y_k,k) \neq  f(\lim_{k \rightarrow \infty} \tau(q(y_0,0),\dots, q(y_k,k))) = f(\lim_{k \rightarrow \infty} \tau'(y_0,\dots y_k)).$$
The $\neq$ follows from having assumed $\tau$ winning strategy for Player I in $G(f)$, and the last equality instead comes directly from having defined $\tau'(y_0,\dots, y_k)$ as $\tau(q(y_0,0),\dots, q(y_k,k)) \cap K$. 

Hence $(y_k)_{k\in\mathbb{N}}$ does not converge to $f(\lim_{k \rightarrow \infty} \tau'(y_0,\dots y_k))$, and $\tau'$ wins the run. Since $(y_k)_{k\in\mathbb{N}}$ was an arbitrary play of Player II, $\tau'$ is a winning strategy for Player I in $G(f{\upharpoonright}  K)$.
\end{proof}

\begin{remark}\label{choice}
The careful reader may have noticed that in this section we didn't use the axiom of choice, or even the axiom of dependent choice, in their full potential. Indeed, all the proofs contained or cited in this section go through assuming only $\mathsf{AC}_\omega(\mathbb{R})$, the axiom of countable choice over the reals: ``Every countable family of nonempty subsets of $\mathbb{R}$ has a choice function".
\end{remark}

\section{On the determinacy of $G(f)$}\label{sec:3}

Recall that a two-player game $G$ is \emph{determined} if either Player I or Player II has a winning strategy. Carroy \cite{Carroy1}*{Theorem 4.1} proved that Duparc's eraser game $G_e(f)$, which characterizes Baire class $1$ functions from and into $\mathbb{N}^\mathbb{N}$, is determined for every function $f$, and used this determinacy result to give a new game-theoretic proof of Baire's grand theorem restricted to functions between $0$-dimensional Polish spaces \cite{Carroy1}*{Theorem 4.6}. On the other hand, Kiss \cite{Kiss}*{\S 2} used  Baire's grand theorem to prove the determinacy of his game. Our game $G(f)$ is a further generalization of both these games, and, again, a strong relationship between its determinacy and Baire's grand theorem emerges as a direct corollary of the two main theorems of the previous section. Let us introduce the following statement, which is the same as \hyperref[BGT]{Baire's grand theorem} with the hypothesis on the domain weakened from Polish to separable metrizable:
\begin{equation}
  \tag{$\mathsf{GBT}$}\label{GBT}
  \parbox{\dimexpr\linewidth-10em}{%
    \strut
    For all $X,Y$ separable metrizable spaces and $f: X\rightarrow Y$, $f$ is Baire class $1$ if and only if $f{\upharpoonright} K$ has a point of continuity for every compact $K\subseteq X$.
    \strut
  }
\end{equation}

The following is a direct corollary of Theorems \ref{PlayerII} and \ref{PlayerI}.

\begin{corollary}\label{equiv1}
    The following are equivalent:
    \begin{enumerate}[label=\arabic*)]
    \itemsep0em
    \item $G(f)$ is determined for every $f$. 
    \item $\mathsf{GBT}$
\end{enumerate}
\end{corollary}

But unlike Duparc's and Kiss' games, ours is not determined in general, as the next folklore proposition shows.

\begin{proposition}\label{ZFC}
$(\mathsf{AC})$ $\mathsf{GBT}$ is false.
\end{proposition}
\begin{proof}
Under the axiom of choice, there exists a set of reals with cardinality of the continuum that does not contain any uncountable closed set. Let $X$ be such a set. Since the family of the $F_\sigma$ subsets of a second countable space has at most the cardinality of the continuum, it follows from Cantor's theorem that there must be a subset $A\subset X$ which is not an $F_\sigma$ subset of $X$.

The function $\mathbbm{1}_A : X\rightarrow 2$, with $\mathbbm{1}_A(x) = 1$ iff $x \in A$, is not Baire class $1$, as $A$ is not an $F_\sigma$ subset of $X$. Nonetheless, we claim that $\mathbbm{1}_A {\upharpoonright}K$ has a point of continuity for every compact $K\subseteq X$. Fix a compact $K\subseteq X$, then $K$ needs to be countable, as we have assumed $X$ not to contain any uncountable closed set. But then $K$, being a countable and compact subset of $\mathbb{R}$, has an isolated point, which is, in particular, a continuity point of $\mathbbm{1}_A {\upharpoonright} K$.
\end{proof}

Note that the same argument shows that $G(\mathbbm{1}_A)$ is undetermined. 

A separable metrizable space is (\emph{absolutely}) \emph{analytic} precisely when it is the continuous image of a Polish space.  Gerlits and Laczkovich \cite{GrandBaire}*{Theorem 2} showed that \hyperref[BGT]{Baire's grand theorem} holds if the domain is assumed only to be an absolutely analytic metrizable space --- actually, they stated this generalization for real functions, but their argument goes through  assuming only separable metrizable codomains. From the theorems of the previous section, it follows that the game $G(f)$ is determined for every function $f$ with analytic domain. 

We cannot hope to extend \textit{tout court} this determinacy result to functions with co-analytic domains, where a separable metrizable space is said to be \emph{co-analytic} if it is homeomorphic to the complement of an analytic set in a Polish space. In fact, the existence of a co-analytic set of cardinality of the continuum that doesn't contain any uncountable closed set is consistent with $\mathsf{ZFC}$ --- in particular, it follows from $V=L$, see \cite{Kanamori}*{Theorem 13.12} --- and the example defined in Proposition \ref{ZFC} would give us a function $f$ with separable metrizable co-analytic domain witnessing the failure of $\mathsf{GBT}$ and the undeterminacy of $G(f)$.

We now focus on $\mathsf{GBT}$, which, by Proposition \ref{ZFC}, is inconsistent with $\mathsf{AC}$. We first introduce a couple of statements coming from Descriptive Set Theory that are strictly related to $\mathsf{GBT}$. We recall that, given three sets $A,B,S$ we say that $S$ \emph{separates $A$ from $B$} if $A \subseteq S$ and $B \cap S = \emptyset$.
\begin{equation}
  \tag{$\mathsf{HSP}$}
  \parbox{\dimexpr\linewidth-10em}{%
    \strut
    For every disjoint $A,B\subseteq \mathbb{N}^\mathbb{N}$ such that there is no $F_\sigma$ set separating $A$ from $B$, there is a Cantor set $\mathcal{C}\subseteq A\cup B$ with $\mathcal{C}\cap B$ countable and dense in $\mathcal{C}$.
    \strut
  }
\end{equation}
\noindent where $\mathsf{HSP}$ stands for Hurewicz' Separation Property. The fact that the trace of $B$ on $\mathcal{C}$ (i.e. $B\cap \mathcal{C}$) is countable and dense not only means that $B\cap \mathcal{C}$ is $F_\sigma$ in $\mathcal{C}$, but also that it is $F_\sigma$-complete \cite{CDST}*{Theorem 22.10}. $\mathsf{HSP}$ is known to hold under $\mathsf{AD}$ \citelist{\cite{CDST}*{\S 21.F} \cite{Carroy2}*{Theorem 4.2}}.
\begin{fact}\label{large}
$\mathsf{HSP}$ is a strong statement, in the sense that $\mathsf{HSP}+\mathsf{DC}$ is equiconsistent with the existence of an inaccessible cardinal. Indeed, $\mathsf{HSP+DC}$ implies $\mathsf{PSP}$, the perfect set property for every subset of $\mathbb{N}^\mathbb{N}$, and it is well-known that $\mathsf{PSP+DC}$ implies the consistency of an inaccessible cardinal \cite{Kanamori}*{Propositions 11.4 and 11.5}; on the other hand,  Todor\v{c}evi\'{c} and Di Prisco \cite{Todorcevic}*{Theorem 4.1} proved that $\mathsf{HSP}$ holds in Solovay's model, hence the equiconsistency.
\end{fact}

Consider now this seemingly weaker statement:
\begin{equation}
  \tag{$\mathsf{WHSP}$}
  \parbox{\dimexpr\linewidth-10em}{%
    \strut
    For every disjoint $A,B\subseteq \mathbb{N}^\mathbb{N}$ such that there is no $F_\sigma$ set separating $A$ from $B$, there is a Cantor set $\mathcal{C}\subseteq A\cup B$ with $\mathcal{C}\cap A$ dense and codense in $\mathcal{C}$.
    \strut
  }
\end{equation}
This statement is clearly a consequence of $\mathsf{HSP}$, but it doesn't tell us anything about the definability of the trace of $A$ or $B$ on $\mathcal{C}$.

\begin{theorem}\label{equiv2}
The following are equivalent:
\begin{enumerate}[label=\arabic*)]
    \item $\mathsf{GBT}$
    \item $\mathsf{WHSP}$
\end{enumerate}
\end{theorem}

\begin{proof}
$1) \Longrightarrow 2)$: let $A,B\subseteq \mathbb{N}^\mathbb{N}$ be disjoint subsets of the Baire space such that $A$ cannot be separated from $B$ by an $F_\sigma$ set. Equivalently, $A$ is not $F_\sigma$ with respect to the relative topology on $A\cup B$. Therefore, the function $\mathbbm{1}_A: A\cup B \rightarrow 2$, with $\mathbbm{1}_A(x) = 1$ iff $x \in A$, is not Baire class $1$, and by $\mathsf{GBT}$ there exists a compact set $ K\subseteq A \cup B$ such that $\mathbbm{1}_A{\upharpoonright} K$ has no points of continuity. This means that $A\cap  K$ is both dense and codense in $ K$, as otherwise $\mathbbm{1}_A$ would have a point of continuity. Finally, notice that $ K$, being a compact and perfect subset of $\mathbb{N}^\mathbb{N}$, is actually a Cantor set \cite{CDST}*{Theorem 7.4}. Hence $\mathsf{WHSP}$ holds. \vspace{1em}

$2)\Longrightarrow 1)$: let $X,Y$ be separable metrizable spaces and $f:X\rightarrow Y$ a function which is not Baire class $1$. Every Polish space is the image of $\mathbb{N}^\mathbb{N}$ by a continuous and closed map \cite{Engelking}. As every separable metrizable space embeds into a Polish space, there exists a closed and continuous surjection $g: X' \rightarrow X$ for some $X' \subseteq \mathbb{N}^\mathbb{N}$. Since the image of a closed set by a closed function is still closed by definition, images of $F_\sigma$ sets by a closed function remain $F_\sigma$. Therefore, the function $h = f \circ g : X' \rightarrow Y$ is still not Baire class $1$.

As $h$ is not Baire class $1$, there is an open set $V\subseteq Y$ such that $h^{-1}(V)$ is not an $F_\sigma$ set of $X'$. Fix such $V$ and also fix a sequence of closed sets $(F_n)_{n\in\mathbb{N}}$ such that $V = \bigcup_n F_n$. It must be the case that, for some $n, \ h^{-1}(F_n)$ is not separable from $h^{-1}(Y\setminus V)$ by an $F_\sigma$ set, as otherwise $h^{-1}(V)$  would be a countable union of $F_\sigma$ sets, which is still $F_\sigma$. Fix such an $n$, then, by $\mathsf{WHSP}$, there is a Cantor set $\mathcal{C}\subseteq X'$ with both $h^{-1}(F_n) \cap \mathcal{C}$ and $h^{-1}(Y\setminus V) \cap \mathcal{C}$ dense in $\mathcal{C}$.

By continuity of $g$, the set $g[\mathcal{C}]$ is compact in $X$ and $f^{-1}(F_n) \cap g[\mathcal{C}], f^{-1}(Y\setminus V) \cap g[\mathcal{C}]$ are both dense in $g[\mathcal{C}]$.

We claim that the function $f{\upharpoonright } g[\mathcal{C}]$ has no points of continuity. Take an $x \in g[\mathcal{C}]$, and fix two sequences $(x_k)_{k\in\mathbb{N}} \subset f^{-1}(F_n) \cap g[\mathcal{C}]$, $(x'_k)_{k\in\mathbb{N}} \subset f^{-1}(Y\setminus V) \cap g[\mathcal{C}]$ converging to $x$. Such sequences exist because $f^{-1}(F_n) \cap g[\mathcal{C}]$ and $f^{-1}(Y\setminus V) \cap g[\mathcal{C}]$ are both dense in $ g[\mathcal{C}]$. If the sequences $(f(x_k))_{k\in\mathbb{N}}$ and  $(f(x'_k))_{k\in\mathbb{N}}$ converged in $Y$, they would converge in $F_n$ and in $Y\setminus V$, respectively, as both these sets are closed.  Thus, even if their limits were to exist, they could not coincide. In particular, $x$ is not a point of continuity of $f{\upharpoonright } g[\mathcal{C}]$. Since $x \in g[\mathcal{C}]$ was arbitrary, we have that no $x \in g[\mathcal{C}]$ is a continuity point of $f{\upharpoonright } g[\mathcal{C}]$.

Given a function $f:X\rightarrow Y$ between separable metrizable spaces which is not Baire class $1$, we have found a compact $ K\subseteq X$ such that $f{\upharpoonright }  K$ has no points of continuity. On the other hand, if $f:X\rightarrow Y$ is Baire class $1$, then the classical argument used in the proof of \hyperref[BGT]{Baire's grand theorem} shows that $f{\upharpoonright} K$ has a point of continuity for every compact $K\subseteq X$ (e.g. see \cite{CDST}*{Theorem 24.15}), with no need to invoke $\mathsf{WHSP}$. Hence $\mathsf{GBT}$ holds.
\end{proof}

As $\mathsf{HSP}$, and in particular $\mathsf{WHSP}$, holds under $\mathsf{AD}$ \citelist{\cite{CDST}*{\S 21.F}\cite{Carroy2}*{Theorem 4.2}} and in Solovay's model \cite{Todorcevic}*{Theorem 4.1}, we can say the same of $\mathsf{GBT}$ and the full determinacy of our game, by Theorem \ref{equiv2} and Corollary \ref{equiv1}. However, the precise consistency strength of these three statements ($+\mathsf{DC}$) is still unknown. 

$\mathsf{WHSP}$, compared to $\mathsf{HSP}$, seems to be weak enough to be proved consistent relative to $\mathsf{ZF}$, with no large cardinals needed (see Fact \ref{large}). Hence, the following conjecture.
\begin{conj}
$\mathsf{GBT+DC}$ is consistent relative to $\mathsf{ZF}$.
\end{conj}

Lastly, notice that the \hyperref[definition]{definition} of our game does not rely on the separability or the metrizability of the function's domain and codomain, and it would make perfect sense to study it on broader classes of functions. Future research can shed  light on how our game behaves on the class of functions with metrizable (not necessarily separable) domains and separable metrizable codomains. Would our results of \hyperref[sec:2]{Section 2} still hold? How much choice would be needed to prove them?


\begin{bibdiv}
\begin{biblist}

\bib{Kiss}{article}{
   author={Kiss, Viktor},
   title={A game characterizing Baire class 1 functions},
   journal={J. Symb. Log.},
   volume={85},
   date={2020},
   number={1},
   pages={456--466},
   issn={0022-4812}
}

\bib{Carroy1}{article}{
   author={Carroy, Rapha\"{e}l},
   title={Playing in the first Baire class},
   journal={MLQ Math. Log. Q.},
   volume={60},
   date={2014},
   number={1-2},
   pages={118--132},
   issn={0942-5616}
}


\bib{Carroy2}{article}{
   author={Carroy, Rapha\"{e}l},
   author={Miller, Benjamin D.},
   author={Soukup, D\'{a}niel T.},
   title={The open dihypergraph dichotomy and the second level of the Borel
   hierarchy},
   conference={
      title={Trends in set theory},
   },
   book={
      series={Contemp. Math.},
      volume={752},
      publisher={Amer. Math. Soc., [Providence], RI},
   },
   date={2020},
   pages={1--19}
}

\bib{GrandBaire}{article}{
   author={Gerlits, János},
   author={Laczkovich, Miklós},
   title={On barely continuous and Baire $1$ functions},
   conference={
      title={Topology, Vol. II},
      address={Proc. Fourth Colloq., Budapest},
      date={1978},
   },
   book={
      series={Colloq. Math. Soc. J\'{a}nos Bolyai},
      volume={23},
      publisher={North-Holland, Amsterdam-New York},
   },
   date={1980},
   pages={493--499}
}

\bib{CDST}{book}{
   author={Kechris, Alexander S.},
   title={Classical descriptive set theory},
   series={Graduate Texts in Mathematics},
   volume={156},
   publisher={Springer-Verlag, New York},
   date={1995},
   pages={xviii+402},
   isbn={0-387-94374-9}
}

\bib{Wadge}{book}{
   author={Wadge, William Wilfred},
   title={Reducibility and determinateness on the Baire space},
   note={Thesis (Ph.D.)--University of California, Berkeley},
   publisher={ProQuest LLC, Ann Arbor, MI},
   date={1983},
   pages={334}
}

\bib{Duparc}{article}{
   author={Duparc, Jacques},
   title={Wadge hierarchy and Veblen hierarchy. I. Borel sets of finite
   rank},
   journal={J. Symbolic Logic},
   volume={66},
   date={2001},
   number={1},
   pages={56--86},
   issn={0022-4812}
}


\bib{Todorcevic}{article}{
   author={Di Prisco, Carlos Augusto},
   author={Todor\v{c}evi\'{c}, Stevo},
   title={Perfect-set properties in $L({\bf R})[U]$},
   journal={Adv. Math.},
   volume={139},
   date={1998},
   number={2},
   pages={240--259},
   issn={0001-8708}
}

\bib{Engelking}{article}{
   author={Engelking, Ryszard},
   title={On closed images of the space of irrationals},
   journal={Proc. Amer. Math. Soc.},
   volume={21},
   date={1969},
   pages={583--586},
   issn={0002-9939}
}

\bib{Kanamori}{book}{
   author={Kanamori, Akihiro},
   title={The higher infinite},
   series={Springer Monographs in Mathematics},
   edition={2},
   note={Large cardinals in set theory from their beginnings},
   publisher={Springer-Verlag, Berlin},
   date={2003},
   pages={xxii+536},
   isbn={3-540-00384-3}
}

\bib{MottoRos}{article}{
   author={Motto Ros, Luca},
   title={Game representations of classes of piecewise definable functions},
   journal={MLQ Math. Log. Q.},
   volume={57},
   date={2011},
   number={1},
   pages={95--112},
   issn={0942-5616}
}



\bib{Kuratowski}{book} {
    AUTHOR = {Kuratowski, Kazimierz},
     TITLE = {Topologie},
     VOLUME = {1},
     translator = {Jaworowski, Jan W.},
 PUBLISHER = {Accademic Press, New York-London},
      YEAR = {1966}
}


\bib{Camerlo}{article}{
   author={Camerlo, Riccardo},
   author={Duparc, Jacques},
   title={Some remarks on Baire's grand theorem},
   journal={Arch. Math. Logic},
   volume={57},
   date={2018},
   number={3-4},
   pages={195--201},
   issn={0933-5846}
}

\bib{Baire}{book}{
   author={Baire, Ren\'{e}},
   title={Le\c{c}ons sur les fonctions discontinues},
   language={French},
   series={Les Grands Classiques Gauthier-Villars. [Gauthier-Villars Great
   Classics]},
   note={Reprint of the 1905 original},
   publisher={\'{E}ditions Jacques Gabay, Sceaux},
   date={1995},
   pages={viii+65},
   isbn={2-87647-124-8}
}


\bib{Karlova}{article}{
   author={Karlova, Olena},
   title={A generalization of a Baire theorem concerning barely continuous
   functions},
   journal={Topology Appl.},
   volume={258},
   date={2019},
   pages={433--438},
   issn={0166-8641}
}

\bib{Karlova2}{article}{
   author={Karlova, Olena},
   title={On Baire-one mappings with zero-dimensional domains},
   journal={Colloq. Math.},
   volume={146},
   date={2017},
   number={1},
   pages={129--141},
   issn={0010-1354}
}



\bib{MR3819690}{article}{
   author={Po\v{s}ta, Petr},
   title={Dirichlet problem and subclasses of Baire-one functions},
   journal={Israel J. Math.},
   volume={226},
   date={2018},
   number={1},
   pages={177--188},
   issn={0021-2172}
}

\bib{MR377482}{article}{
   author={Odell, Edward Wilfred},
   author={Rosenthal, Haskell Paul},
   title={A double-dual characterization of separable Banach spaces
   containing $l^{1}$},
   journal={Israel J. Math.},
   volume={20},
   date={1975},
   number={3-4},
   pages={375--384},
   issn={0021-2172}
}

\bib{Hausdorff}{book}{
   author={Hausdorff, Felix},
   title={Set theory},
   edition={2},
   note={Translated from the German by John R. Aumann et al},
   publisher={Chelsea Publishing Co., New York},
   date={1962},
   pages={352}
}


\bib{Andretta}{article}{
   author={Andretta, Alessandro},
   title={More on Wadge determinacy},
   journal={Ann. Pure Appl. Logic},
   volume={144},
   date={2006},
   number={1-3},
   pages={2--32},
   issn={0168-0072}
}


\bib{Limsup}{article}{
   author={Elekes, M\'{a}rton},
   author={Flesch, J\'{a}nos},
   author={Kiss, Viktor},
   author={Nagy, Don\'{a}t},
   author={Po\'{o}r, M\'{a}rk},
   author={Predtetchinski, Arkadi},
   title={Games characterizing limsup functions and Baire class 1 functions},
   journal={J. Symb. Log.},
   volume={87},
   date={2022},
   number={4},
   pages={1459--1473},
   issn={0022-4812}
}

\bib{balcerzak}{arXiv}{
  author={Marek Balcerzak},
  author={Tomasz Natkaniec},
  author={Piotr Szuca},
  title={Games characterizing certain families of functions},
  date={2023},
  eprint={2109.05458},
  archiveprefix={arXiv}
}

\end{biblist}
\end{bibdiv}
\end{document}